\documentclass[preprint,12pt]{elsarticle}
\usepackage{latexsym}
\usepackage{amssymb}
\usepackage{amsmath}
\usepackage{amsthm}
\usepackage{verbatim}
\usepackage{mathrsfs}
\numberwithin{equation}{section}

\newtheorem{theorem}{Theorem}[section]
\newtheorem{lemma}{Lemma}[section]
\newtheorem{proposition}{Proposition}[section]
\newtheorem{remark}{Remark}[section]
\newtheorem{definition}{Definition}[section]

\begin{document}

\begin{frontmatter}



\title{Infinitely many periodic solutions for a semilinear wave equation with $x$-dependent coefficients
%\tnoteref{ack}
}
%\tnotetext[ack]{This work is partially supported by NSFC Grants (nos. 11322105 and 11671071).}

\author{Hui Wei}
\ead{weihui01@163.com}
\author{Shuguan Ji\corref{cor}}
\ead{jishuguan@hotmail.com}
\address{School of Mathematics and Statistics and Center for Mathematics and Interdisciplinary Sciences, Northeast Normal University, Changchun 130024, P.R. China}
\cortext[cor]{Corresponding author.}

\begin{abstract}
This paper is devoted to the study of periodic (in time) solutions to an one-dimensional semilinear wave equation with $x$-dependent coefficients under various homogeneous boundary conditions. Such a model arises from the forced vibrations
of a nonhomogeneous string and propagation of seismic waves in nonisotropic media. By combining variational methods with an approximation argument, we prove that there exist infinitely many periodic solutions whenever the period is a rational multiple of the length of the spatial interval. The proof is essentially based on the spectral properties of the wave operator with $x$-dependent coefficients.
\end{abstract}

\begin{keyword}
$\mathbb{Z}_2$-index,  infinitely many periodic solutions, wave equation
\end{keyword}

\end{frontmatter}
% \linenumbers

\section{Introduction}

In this paper, we concern with the existence of infinitely many periodic solutions to the wave equation with $x$-dependent coefficients
\begin{equation}\label{eqa:1.1}
 \rho(x) u_{tt} -  (\rho(x)u_x)_x = \mu\rho(x) u + \rho(x)|u|^{p-1}u,  \ \ t\in \mathbb{R}, \ \ 0<x<\pi,
\end{equation}
with the boundary conditions
\begin{equation}\label{eqa:1.2}
a_1u(t, 0)+b_1 u_x(t, 0)=0, \ \ a_2u(t, \pi)+ b_2 u_x(t, \pi) = 0, \, t\in \mathbb{R},
\end{equation}
and the periodic conditions
\begin{equation}\label{eqa:1.3}
u(t+T,x) = u(t,x), \ u_t(t+T,x) = u_t(t,x),  \ \ t\in \mathbb{R}, \ \ 0<x<\pi.
\end{equation}
Here $a^2_i + b^2_i \neq 0$ for $i=1, 2$, $0<p<1$ and $\mu>0$  is a constant, and $T$ is a rational multiple of $\pi$. For convenience, we write
$$T = 2\pi \frac{a}{b},$$
where $a$, $b$ are relatively prime positive integers.

As stated in \cite{Barbu.(1996), Barbu.(1997)a, Barbu.(1997)b, Ji.(2008), Ji.(2009), Ji.(2016), Ji.(2006), Ji.(2007), Ji.(2011), Ru04, Ru17}, equation (\ref{eqa:1.1}) is a mathematical model to account for the forced vibrations of a bounded nonhomogeneous string and the propagation of seismic waves in nonisotropic media.
More precisely, the vertical displacement $u(t,z)$ at time $t$ and depth $z$ of a plane seismic wave is described by the equation
\begin{equation}
\label{eqa:1.4}
\omega(z) u_{tt} -  (\nu(z)u_z)_z = 0
\end{equation}
with some initial conditions in $t$ and boundary conditions in $z$, where $\omega(z)$ is the rock density and $\nu(z)$ is the elasticity coefficient.
By the change of variable
$$x =\int_0^z \left(\frac{\omega(s)}{\nu(s)}\right)^{1/2} \textrm{d}s,$$
equation (\ref{eqa:1.4}) is transformed into
$$\rho(x) u_{tt} -  (\rho(x)u_x)_x=0,$$
where $\rho=(\omega\nu)^{1/2}$ denotes the acoustic impedance function.

It is well known that when $\rho(x)$ is a non-zero constant, it corresponds to the constant coefficients wave equation.
A great deal of attention has been paid to find periodic solutions of nonlinear wave equations with constant coefficients since 1960s. The related results are \cite{Amann.(1979), B.(1983), B.(1978), Chang.(1981), Craig.(1993), Chen.(2017), R.(1967), R.(1978), R.(1984), T.(2006)} and the references therein. In recent decades, the problem of finding infinitely many periodic solutions of wave equations with constant coefficients has captured much research interest. In \cite{T.1985}, with the aid of variational methods, Tanaka proved that there exist infinitely many periodic solutions for the nonhomogeneous one-dimensional wave equation with the nonlinearity $|u|^{p-1}u $, where $1<p<1+\sqrt{2}$. Later, he \cite{T.1988} extended the result to the case $p>1$ by a delicate calculation. Employing variational methods together with an approximation argument, Ding and Li \cite{D.1998} proved that the homogeneous one-dimensional wave equation with the general nonlinearity which is odd and nondecreasing, and super- or sub-linearities possesses infinitely many periodic solutions. For higher dimensional case, Chen and Zhang \cite{Chen.(2014), Chen.(2016)} considered the wave equation with the nonlinearity $|u|^{p-1}u$ in a ball in $\mathbb{R}^N$ and obtained the existence of infinitely many periodic solutions by using variational methods and an approximation argument, where $0<p<1$. These results are essentially based on the properties of the spectrum of the wave operator.

On the other hand, Barbu and Pavel \cite{Barbu.(1996), Barbu.(1997)a, Barbu.(1997)b} first studied the problem of finding periodic solutions to the nonlinear wave equations with $x$-dependent coefficients. For the nonlinear term satisfying Lipschitz continuous and sub-linear growth, Barbu and Pavel \cite{Barbu.(1997)a} proved the existence and regularity of periodic solutions. Rudakov \cite{Ru04} considered the existence of periodic solutions for such wave equation with power-law growth nonlinearity. Ji and Li \cite{Ji.(2011)} obtained an existence result of periodic solution for $\eta_{\rho}(x)=0$, which actually solves an open problem posted in \cite{Barbu.(1997)a}. Chen \cite{C.2015} applied a global inverse function theorem (see \cite{P.1974}) to prove the existence and uniqueness of periodic solutions for a system of nonlinear wave equation with $x$-dependent coefficients. These papers deal with the problem under the Dirichlet boundary conditions. Compared with the above works, Ji and his collaborators acquired some related results for the general Sturm-Liouville
boundary value problem \cite{Ji.(2008), Ji.(2006)}, and periodic and anti-periodic boundary value problem \cite{Ji.(2009), Ji.(2007)}.
In \cite{W.2009}, by using topological degree methods, Wang and An obtained an existence result on periodic solution of the problem with resonance and the sub-linear nonlinearity under Dirichlet-Neumann boundary conditions.
Recently, with the help of the Leray-Schauder degree theory, Ji et al. \cite{Ji.(2016)} obtained the existence and multiplicity of periodic solutions under the Dirichlet-Neumann boundary conditions. The restriction to such type of boundary value problem can guarantee the compactness of the inverse operator on its range.

In comparison with these works, the aim of this paper is devoted to the existence of infinitely many periodic solutions for the problem \eqref{eqa:1.1}--\eqref{eqa:1.3}.
In order to solve this problem, we construct suitable function space which is called working space here. By combining the $\mathbb{Z}_2$-index theory with  minimax method for even functional, we first work out the problem on some given  subspaces of working space. Then we obtain infinitely many periodic solutions of the problem \eqref{eqa:1.1}--\eqref{eqa:1.3} by a limiting process. Our method is based on a delicate analysis for the asymptotic character of
the spectrum of the wave operator with $x$-dependent coefficients under various homogeneous boundary conditions, and the spectral properties play an essential role in the proof.

In this paper, we make the following assumption:

(A1) $\rho(x) \in H^2(0, \pi)$ satisfies for any $x\in[0,\pi]$,  $0<\rho(x)\leq\beta_0$ and
$$\rho_0 = \textrm{ess}\inf \eta_{\rho}(x) >0,$$
where
$$\eta_{\rho}(x) = \frac{1}{2} \frac{\rho''}{\rho}-\frac{1}{4} \left(\frac{\rho'}{\rho}\right)^2.$$

(A2) Set $\alpha_1=a_1-\frac{b_1}{2}(\frac{\rho'(0)}{\rho(0)})$, $\alpha_2=a_2-\frac{b_2}{2}(\frac{\rho'(\pi)}{\rho(\pi)})$, $\beta_1=-b_1$, $\beta_2=b_2$, and satisfy
$$\alpha_i\geq0, \ \beta_i \geq 0  \ {\rm and} \  \alpha_i +\beta_i>0, \ {\rm for} \ i=1, 2.$$

The rest of this paper is organized as follows. In Sect. \ref{sec:2}, we give some notations and preliminaries such as the definition of weak solution of problem \eqref{eqa:1.1}--\eqref{eqa:1.3} and the asymptotic character of the spectrum of the wave operator with $x$-dependent coefficients under various homogeneous boundary conditions, and the definition of working space. Meanwhile, we characterize the solutions of problem \eqref{eqa:1.1}--\eqref{eqa:1.3} as the critical points of the corresponding variational problem, and state the main results. In Sect. \ref{sec:3}, we prove the bounds of the corresponding functional on some spherical surfaces and some subspaces. In Sect. \ref{sec:4}, with the aid of $\mathbb{Z}_2$-index theory and minimax method, we obtain a sequence of the critical points for the corresponding functional restricted on some given subspaces. In Sect. \ref{sec:5}, by an approximation argument, we prove the main result.
\section{Preliminaries and main results}
\setcounter{equation}{0}
\label{sec:2}

Set $\Omega = (0, T) \times (0, \pi)$, and denote
\begin{eqnarray*}
\Psi = \{\psi \in C^\infty(\Omega) : a_1\psi(t, 0)+b_1 \psi_x(t, 0)=0, \ a_2\psi(t, \pi)+ b_2 \psi_x(t, \pi) = 0, \\
\psi(0,x) = \psi(T,x), \psi_t(0,x) = \psi_t(T,x)\},
\end{eqnarray*}
and
$$L^r(\Omega) = \Big\{ u: \|u\|^r_{L^r(\Omega)} = \int_\Omega |u(t,x)|^r \rho(x) \textrm{d}t \textrm{d}x <\infty\Big\}, \, \, r\geq 1.$$
It is well known that $\Psi$ is dense in $L^r(\Omega)$ for any $r\geq 1$, and $L^2(\Omega)$ is a Hilbert space with the inner product
$$\langle u, v \rangle = \int_\Omega u(t,x)  \overline{v(t,x)} \rho(x)\textrm{d}t \textrm{d}x, \ \forall u, v \in L^2(\Omega).$$

We rewrite (\ref{eqa:1.1})--(\ref{eqa:1.3}) on $\Omega$ in the
 following form
 \begin{eqnarray}
\label{eqa:2-a}
 && \rho(x)u_{tt}-(\rho(x)u_{x})_{x}=\mu\rho(x) u +\rho(x)|u|^{p-1}u, \
 (t,x)\in \Omega,\\
\label{eqa:2-b}
&&a_1u(t, 0)+b_1 u_x(t, 0)=0, \ a_2u(t, \pi)+ b_2 u_x(t, \pi) = 0, \ t\in(0,T), \\
\label{eqa:2-c}
 &&u(0,x)=u(T,x),\ u_{t}(0,x)=u_{t}(T,x), \
x\in(0,\pi).
 \end{eqnarray}
\begin{definition}
A function $u \in L^r(\Omega)$ is said to be a weak solution to problem \eqref{eqa:2-a}--\eqref{eqa:2-c} if
$$\int_\Omega u(\rho\psi_{tt} - (\rho\psi_x)_x){\rm d}t \textrm{\rm d}x - \int_\Omega (\mu u + |u|^{p-1}u)\psi \rho\textrm{\rm d}t \textrm{\rm d}x = 0,\,\, \forall \psi \in \Psi. $$
\end{definition}

In order to look for the periodic solutions of problem (\ref{eqa:2-a})--(\ref{eqa:2-c}), we need to use the following
complete orthonormal system of eigenfunctions $\{\phi_j(t)\varphi_k(x) : j\in \mathbb{Z}, k\in \mathbb{N}\}$ in $L^2(\Omega)$ (see \cite{Y.(1980)}), where
\begin{equation*}
\phi_j(t) = T^{-\frac{1}{2}} e^{i\nu_j t}, \ \nu_j = 2j\pi T^{-1}, \ j \in \mathbb{Z},
\end{equation*}
and $\lambda_k$, $\varphi_k(x)$ are given by the Sturm-Liouville problem
\begin{eqnarray}\label{eqa:2-d}
\begin{split}
&-(\rho(x)\varphi'_k(x))' = \lambda^2_k \rho(x) \varphi_k(x),  \ k\in \mathbb{N},\\
&a_1u(0)+b_1 u_x(0)=0, \\
&a_2u(\pi)+ b_2 u_x( \pi) = 0.
\end{split}
\end{eqnarray}
It is known that $\lambda^2_n$ is increasingly convergent to $+\infty$. Set $z_n(x)=(\rho(x))^{1/2}\varphi_n(x)$, then \eqref{eqa:2-d} can be transformed into the following Sturm-Liouville problem
\begin{eqnarray}\label{eqa:2-e}
\begin{split}
&z''_k(x) + (\lambda^2_k-\eta_{\rho}(x))z_k(x)=0,  \ k\in \mathbb{N},\\
&\alpha_1z(0)-\beta_1 z_x( 0)=0, \\
&\alpha_2z(\pi)+ \beta_2 z_x(\pi) = 0,
\end{split}
\end{eqnarray}
where $\alpha_i, \beta_i$ satisfy the assumption (A2) for $i=1, 2$. In this situation, it is more convenient to study the properties of the eigenvalues $\lambda_k^2$.
Now, we divide \eqref{eqa:2-e} into the following several cases:
$$ \emph{Case 1}: \alpha_1>0, \beta_1=0, \alpha_2>0, \beta_2=0;$$
$$ \emph{Case 2}: \alpha_1>0, \beta_1=0, \alpha_2=0, \beta_2>0;$$
$$ \emph{Case 3}: \alpha_1=0, \beta_1>0, \alpha_2>0, \beta_2=0;$$
$$ \emph{Case 4}: \alpha_1=0, \beta_1>0, \alpha_2=0, \beta_2>0;$$
$$ \emph{Case 5}: \alpha_1>0, \beta_1>0, \alpha_2>0, \beta_2>0.$$

In what follows, we shall deal with problem \eqref{eqa:2-a}--\eqref{eqa:2-c} according to the five cases. It is well known that Case 1 is called Dirichlet boundary conditions. Case 2 and 3 are called Dirichlet-Neumann boundary conditions, and by the transformation $\tilde{x}=\pi -x$ we can prove that Case 2 is equivalent to Case 3. Thus, we only deal with the Case 2 here. Case 4 is called Neumann boundary conditions and Case 5 is called  general boundary conditions.

Define the linear operator $L_0$ by
$$L_0 \psi = \rho^{-1}\left(\rho \psi_{tt} - (\rho \psi_x)_x\right), \ \forall \psi \in \Psi,$$
and denote its extension in $L^2(\Omega)$ by $L$. It is known that $L$ is a selfadjoint operator (see \cite{Barbu.(1997)a, Ji.(2008)}), and $u\in L^2(\Omega)$ is a weak solution of problem (\ref{eqa:2-a})--(\ref{eqa:2-c}) if and only if
$$Lu=\mu u +|u|^{p-1}u.$$
Moreover, it is easy to see that the eigenvalues of $L$ have the form
$\lambda_{jk} = \lambda_k^2-\nu_j^2$.
Denote the set of eigenvalues of operator $L$ by
$$\Lambda(L)=\{\lambda_{jk}: \lambda_{jk}=\lambda_k^2-\nu_j^2\}.$$

Furthermore, for above Case 1 and Case 2--5, Barbu and Pavel \cite{Barbu.(1997)a} and Ji \cite{Ji.(2008)} have  characterize the asymptotic formulas of $\lambda_n^2$ respectively. Then, we will make use of the asymptotic formulas to study the properties of the eigenvalues
$$\lambda_{jk}=\lambda_k^2-\nu_j^2.$$

\begin{lemma}[\cite{Barbu.(1997)a}] \label{lem:2.1}
Assume that $\rho(x)$ satisfies (A1), then
the eigenvalues of problem \eqref{eqa:2-e} with the Dirichlet boundary conditions (i.e., Case 1) have the form
$$\lambda_k = k +\theta_k \ with \ \theta_k \rightarrow 0 \ as \ k \rightarrow \infty,$$
where
\begin{equation}\label{eqa:2-f}
0<\frac{\rho_2}{k} \leq \sqrt{k^2 + \rho_0}-k \leq \theta_k \leq \sqrt{k^2 + \rho_1}-k \leq \frac{\rho_1}{2k}, \ \ for \ \  k \geq 1,
\end{equation}
and $\rho_1 = \frac{2}{\pi} \int^{\pi}_0 \eta_{\rho}(x) \textrm{d}x$, $\rho_2 = \sqrt{\rho_0 +1} -1$.
\end{lemma}

By above lemma, we can easily obtain the following lemma.
\begin{lemma}\label{lem:2.2}
Let assumption $(A1)$ hold. Then\\
(i) $L$ has at least one essential spectral point, and all of them belong to $[2\rho_2, \rho_1]${\rm ;}\\
(ii) If $\lambda \in \Lambda(L)$ and $\lambda \notin [2\rho_2, \rho_1]$, then $\lambda$ is isolated and its multiplicity is finite.
\end{lemma}
\begin{proof}
By Lemma \ref{lem:2.1}, the eigenvalues of operator $L$ can be rewritten as
$$\lambda_{jk} = \lambda_k^2-\nu_j^2= a^{-2}(ka + \theta_ka-jb)(ka + \theta_ka+jb).$$
Thus,  when $ka\neq |j|b$, it is easy to verify that $|\lambda_{jk}| \rightarrow \infty$ as $j, k\rightarrow  \infty$.  On the other hand, when  $ka = |j|b$, by \eqref{eqa:2-f} we have
\begin{equation*}
2\rho_2\leftarrow(\frac{\rho_2}{k})^2 + 2\rho_2 \leq \lambda_{jk}= \theta_k (2k+\theta_k) \leq \rho_1+(\frac{\rho_1}{2k})^2\rightarrow \rho_1,
\end{equation*}
as $k \rightarrow \infty$.
Therefore, $\Lambda(L)$  has at least one accumulation point in $[2\rho_2, \rho_1]$. Moreover,  $\lambda$ is isolated and its multiplicity is finite when $\lambda \in \Lambda(L)$ and $\lambda \notin [2\rho_2, \rho_1]$. The proof is completed.
\end{proof}

For the case 1,  we obtain the main result.
\begin{theorem}\label{th:2.1}
Assume that $0<p<1$, and $\mu \notin \Lambda(L)$ and satisfies $\mu >\rho_1$. If the assumption $(A1)$ holds, then the problem \eqref{eqa:2-a}--\eqref{eqa:2-c} with Dirichlet boundary conditions (i.e. Case 1) has infinitely many periodic weak solutions $u_l$ satisfying
$$\int_\Omega |u_l|^{p+1} \rho \textrm{\rm d}t \textrm{\rm d}x \rightarrow 0, \ \ { as} \ \ l\rightarrow \infty.$$
\end{theorem}

For the case 2 or 3,  we have the following results.
\begin{lemma}[\cite{Ji.(2008)}] \label{lem:2.a}
Assume that $\rho(x)$ satisfies (A1), then
the eigenvalues of problem \eqref{eqa:2-e} with Dirichlet-Neumann boundary conditions (i.e., Case 2 or 3) have the form
$$\lambda_k = k +1/2 +\theta_k \ with \ \theta_k \rightarrow 0 \ as \ k \rightarrow \infty,$$
where
\begin{equation*}
0<\frac{\rho_3}{2k+1} \leq  \theta_k \leq \frac{\rho_1}{\sqrt{(k+1/2)^2+\rho_1} + k+1/2} \leq \frac{\rho_1}{2k +1},
\end{equation*}
and $\rho_3 = \frac{\rho_0}{1+\rho_0}$, $\rho_1 = \frac{2}{\pi} \int^{\pi}_0 \eta_{\rho}(x) \textrm{d}x$.
\end{lemma}

By a similar proof as in Lemma \ref{lem:2.2}, we obtain the following Lemma \ref{lem:2.b}, Lemma \ref{lem:2.d} and Lemma \ref{lem:2.f}. Thus, their proofs we omit here.
\begin{lemma}\label{lem:2.b}
Let assumption $(A1)$ hold. Then\\
(i) $L$ has at least one essential spectral point, and all of them belong to $[\rho_3, \rho_1]${\rm ;}\\
(ii) If $\lambda \in \Lambda(L)$ and $\lambda \notin [\rho_3, \rho_1]$, then $\lambda$ is isolated and its multiplicity is finite.
\end{lemma}

\begin{theorem}\label{th:2.2}
Assume that $0<p<1$, and $\mu \notin \Lambda(L)$ and satisfies $\mu >\rho_1$. If the assumption $(A1)$ holds, then the problem \eqref{eqa:2-a}--\eqref{eqa:2-c} with Dirichlet-Neumann boundary conditions (i.e., Case 2 or 3) has infinitely many periodic weak solutions $u_l$ satisfying
$$\int_\Omega |u_l|^{p+1} \rho \textrm{\rm d}t \textrm{\rm d}x \rightarrow 0, \ \ { as} \ \ l\rightarrow \infty.$$
\end{theorem}

Similarly, for the case 4, we have the following results.
\begin{lemma}[\cite{Ji.(2008)}] \label{lem:2.c}
Assume that $\rho(x)$ satisfies (A1), then
the eigenvalues of problem \eqref{eqa:2-e} with Neumann boundary conditions (i.e., Case 4) have the form
$$\lambda_k = k +\theta_k \ with \ \theta_k \rightarrow 0 \ as \ k \rightarrow \infty,$$
where
\begin{equation*}
0<\frac{\rho_2}{k} \leq  \theta_k \leq \sqrt{k^2+\rho_1}-k \leq \frac{\rho_1}{2k}, \ \ for \ \ k\geq 1,
\end{equation*}
and $\rho_2 = \sqrt{1+\rho_0}-1$, $\rho_1 = \frac{2}{\pi} \int^{\pi}_0 \eta_{\rho}(x) \textrm{d}x$.
\end{lemma}

\begin{lemma}\label{lem:2.d}
Let assumption $(A1)$ hold. Then\\
(i) $L$ has at least one essential spectral point, and all of them belong to $[2\rho_2, \rho_1]${\rm ;}\\
(ii) If $\lambda \in \Lambda(L)$ and $\lambda \notin [2\rho_2, \rho_1]$, then $\lambda$ is isolated and its multiplicity is finite.
\end{lemma}

\begin{theorem}\label{th:2.3}
Assume that $0<p<1$, and $\mu \notin \Lambda(L)$ and satisfies $\mu >\rho_1$. If the assumption $(A1)$ holds, then the problem \eqref{eqa:2-a}--\eqref{eqa:2-c} with Neumann boundary conditions (i.e., Case 4) has infinitely many periodic weak solutions $u_l$ satisfying
$$\int_\Omega |u_l|^{p+1} \rho \textrm{\rm d}t \textrm{\rm d}x \rightarrow 0, \ \ { as} \ \ l\rightarrow \infty.$$
\end{theorem}

For the case 5, we have the following results.
\begin{lemma}[\cite{Ji.(2008)}] \label{lem:2.e}
Assume that $\rho(x)$ satisfies (A1), then there exists a constant $N_0 >1$ such that
the eigenvalues of problem \eqref{eqa:2-e} with the general boundary conditions (i.e., Case 5) have the form
$$\lambda_k = k +\theta_k \ with \ \theta_k \rightarrow 0 \ as \ k \rightarrow \infty,$$
where
\begin{equation*}
0<\frac{\rho_2}{k} \leq  \theta_k \leq \sqrt{k^2+2\rho_4} -k \leq \frac{\rho_4}{k}, \ \ for \ \ k\geq N_0,
\end{equation*}
and
$$\rho_2 = \sqrt{1+\rho_0}-1, \ \ \rho_4 = \frac{1}{\pi} \Big(\frac{\alpha_1}{\beta_1}+\frac{\alpha_2}{\beta_2}+1+\int^{\pi}_0 \eta_{\rho}(x) \textrm{d}x \Big).$$
\end{lemma}

\begin{lemma}\label{lem:2.f}
Let assumption $(A1)$ hold. Then\\
(i) $L$ has at least one essential spectral point, and all of them belong to $[2\rho_2, 2\rho_4]${\rm ;}\\
(ii) If $\lambda \in \Lambda(L)$ and $\lambda \notin [2\rho_2, 2\rho_4]$, then $\lambda$ is isolated and its multiplicity is finite.
\end{lemma}

\begin{theorem}\label{th:2.4}
Assume that $0<p<1$, and $\mu \notin \Lambda(L)$ and satisfies $\mu >2\rho_4$. If the assumption $(A1)$  holds, then the problem \eqref{eqa:2-a}--\eqref{eqa:2-c} with the general boundary conditions (i.e., Case 5) has infinitely many periodic weak solutions $u_l$ satisfying
$$\int_\Omega |u_l|^{p+1} \rho \textrm{\rm d}t \textrm{\rm d}x \rightarrow 0, \ \ { as} \ \ l\rightarrow \infty.$$
\end{theorem}

\begin{remark}
In order to acquire the solutions of problem \eqref{eqa:2-a}--\eqref{eqa:2-c} via variational methods, we construct a function space $E$ (which is called working space here and will be given later) in which critical point theory can be applied.

On the other hand, through a careful observation, the conditions in  Theorem \ref{th:2.1}--\ref{th:2.3} are the same. Hence we treat them as the same one. Moreover, the only difference between Theorem \ref{th:2.1} and Theorem \ref{th:2.4} lies in the range of the constant $\mu$. More precisely, in Theorem \ref{th:2.1} we require $\mu > \rho_1$, but $\mu > 2\rho_4$ in Theorem \ref{th:2.4}. The condition $\mu > \rho_1$ or $\mu > 2\rho_4$ can make sure that the working space $E$ can be well defined, which play an essential role in our proof. Therefore, by above discussion we only give the proof of  Theorem \ref{th:2.1} here and by the same argument we can also prove Theorem \ref{th:2.4}.
\end{remark}

In what follows, we always assume that $\mu \notin \Lambda(L)$ and $\mu >\rho_1$, then  there exists a constant $\delta > 0$ such that
\begin{eqnarray}\label{eqa:2.3}
|\lambda_{jk} - \mu|\geq \delta >0, \ j \in \mathbb{Z}, k\in  \mathbb{N}^+.
\end{eqnarray}

For $u \in L^2(\Omega)$, we rewrite $u(t,x) = \sum\limits_{j,k} \alpha_{jk} \phi_j(t)\varphi_k(x)$, where  $\alpha_{jk}$ are the Fourier coefficients. Define function space
$$E= \Big\{u \in L^2(\Omega)\mid  \|u\|^2_E = \sum\limits_{j,k}|\lambda_{jk} - \mu| |\alpha_{jk}|^2 < \infty\Big\},$$
which is called the working space here. The estimate \eqref{eqa:2.3} shows that $\|\cdot\|_E$ is a  norm on $E$.  Furthermore, the space $E$ is Hilbert space equipped with the norm $\|\cdot\|_E$.

Since $\|u\|^2_{L^2(\Omega)} = \sum\limits_{j,k} |\alpha_{jk}(u)|^2$, from \eqref{eqa:2.3}, we have
\begin{eqnarray}\label{eqa:2.4}
\|u\|^2_{L^2(\Omega)} \leq \delta^{-1} \|u\|^2_E
\end{eqnarray}
which implies the continuous embedding $E \hookrightarrow L^2(\Omega)$. Moreover, for $1 \leq r \leq 2$, the continuous embedding $L^2(\Omega) \hookrightarrow L^r(\Omega)$ implies that there exists a constant $C=C(r)$ such that
\begin{eqnarray}\label{eqa:2.5}
\|u\|_{L^r(\Omega)} \leq C \|u\|_E.
\end{eqnarray}

Let $f(u)= |u|^{p-1}u$ and $F(u) = \int^u_0 f(s) \textrm{d}s = \frac{1}{p+1}|u|^{p+1}$. It is easy to see $f$ is odd,  then $F$ is even. Define the functional
\begin{eqnarray}\label{eqa:2.6}
\Phi(u)  = -\frac{1}{2}\langle(L-\mu)u, u\rangle + \int_\Omega F(u) \rho\textrm{d}t \textrm{d}x, \ \ \forall u  \in E.
\end{eqnarray}
Thus, $\Phi$ is an even $C^1$ functional on $E$ and
\begin{eqnarray}\label{eqa:2.7}
\langle \Phi'(u), v \rangle = -\langle(L-\mu)u, v\rangle + \int_\Omega f(u)v \rho\textrm{d}t \textrm{d}x, \ \ \forall u, \ v\in E.
\end{eqnarray}

Consequently, $u$ is a weak solution of problem \eqref{eqa:2-a}--\eqref{eqa:2-c} if and only if $\Phi'( u) = 0$. Thus, we can characterize the solutions of problem \eqref{eqa:2-a}--\eqref{eqa:2-c} as the critical points of the functional $\Phi$.
In addition, from the proof of Lemma \ref{lem:2.2}, it's not difficult to see that $\Phi$ is neither bounded from above nor from below, which shows that we can't obtain the critical points of $\Phi$ by a simple minimization or maximization. In what follows, by variational methods together with an approximation argument, we prove that there exist infinitely many periodic solutions for the problem \eqref{eqa:2-a}--\eqref{eqa:2-c}.

%In order to construct approximate solutions on a finite subspace of $E$, we need to decompose the space $E$.
%In what follows, we will investigate the asymptotic behavior of $\zeta_{l}$.
\section{Bounds of the functional $\Phi$}
\label{sec:3}
In this section, we concern with the bounds of $\Phi$ on some  spherical surfaces and some subspaces, which can help us to distinguish the critical points by the different critical values. Thus, we firstly need to decompose the space $E$ into some suitable subspaces.

Recalling that $\mu \notin \Lambda(L)$ and $\mu >\rho_1$,  the working space $E$ can be decomposed as a direct sum $E = E^+\oplus E^-$, where
\begin{displaymath}
\begin{array}{lll}
E^+= \Big\{u \in E\mid  \lambda_{jk}>\mu,  \ j \in \mathbb{Z}, k\in  \mathbb{N}^+ \Big\},\\
E^-= \Big\{u \in E\mid  \lambda_{jk}<\mu, \ j \in \mathbb{Z}, k\in  \mathbb{N}^+\Big\}.
\end{array}
\end{displaymath}

Denote
$$E_0 = \Big\{u \in L^2(\Omega) \mid  \ k a= |j|b,  \ j \in \mathbb{Z}, k\in  \mathbb{N}^+\Big\}.$$

For any $u\in E^-$, we write $u = \sum\limits_{\lambda_{jk}<\mu} \alpha_{jk} \phi_j\varphi_k$, then
\begin{equation}\label{eqa:2.a}
\langle (L - \mu)u, u \rangle= - \sum\limits_{\lambda_{jk}<\mu} |\lambda_{jk} - \mu| |\alpha_{jk}|^2  = -\|u\|^2_E.
\end{equation}
Moreover, for any $u \in E^+$, by a similar calculation we have
\begin{equation}\label{eqa:2.b}
\langle (L - \mu)u, u \rangle = \|u\|^2_E.
\end{equation}
\begin{remark}\label{rem:2.1}
By Lemma {\upshape\ref{lem:2.2}}, we have $E_0$ is an infinite dimensional space spanned by the eigenfunctions $\phi_j\varphi_k$ for $k a= |j|b$. Moreover, it is easy to see $\dim (E^+ \cap E_0) <\infty$ and $\dim (E^- \cap E_0) = \infty$.
\end{remark}
For any $l, m\in \mathbb{N}^+$, let
$$W_m= {\rm span}\Big\{\phi_j\varphi_k \mid -m\leq j \leq m, \ 1\leq k\leq m \Big\},$$
and
\begin{equation*}
E^m= \Big(W_m \cap E^-\Big)\oplus E^+, \ \ E_l= E^-\oplus\Big(W_l \cap E^+\Big),
\end{equation*}
and
$$E^m_l= E^m\cap E_l.$$
Obviously, $E^m \subset E^{m+1}$ and $E=\bigcup_{m\in\mathbb{N}^+} E^m$, and $E^m_l$ is a finite dimensional space.

\begin{lemma}\label{lem:3.2}
There exist $\sigma_l$, $R_l>0$ such that
$$\Phi(u)\geq \sigma_l, \ \ \forall u\in E_{l+1}\cap S_{R_l},$$
where $S_{R_l} = \{ u \in E \mid \|u\|_E = R_l\}$.
\end{lemma}
\begin{proof}
For $u \in E_{l+1}= E^-\oplus (W_{l+1} \cap E^+)$, split $u = u^+ +u^-$, where $u^+\in W_{l+1} \cap E^+$ and $u^- \in E^-$.

Since $\dim (W_{l+1} \cap E^+) < \infty$, we have all norms of $u^+$ in $W_{l+1} \cap E^+$ are equivalent. Moreover, since the projection mapping $P : L^{p+1}(\Omega) \rightarrow W_{l+1} \cap E^+$ is bounded, then there exists a constant $C_0>0$ such that
\begin{equation}\label{eqa:3.5}
\|u^+\|^2_E \leq C_0 \|u\|^2_{p+1}.
\end{equation}

From \eqref{eqa:2.6}, we have
\begin{eqnarray}\label{eqa:3.6}
\Phi(u)  &=& -\frac{1}{2}\langle(L-\mu)u, u\rangle + \frac{1}{p+1}\int_\Omega |u|^{p+1} \rho\textrm{d}t \textrm{d}x \nonumber \\
&=& -\frac{1}{2}\|u^+\|^2_E +\frac{1}{2}\|u^-\|^2_E + C_0 \|u\|^2_{L^{p+1}(\Omega)}+ \frac{1}{p+1}\int_\Omega |u|^{p+1} \rho\textrm{d}t \textrm{d}x -C_0 \|u\|^2_{L^{p+1}(\Omega)}\nonumber \\
&=& I_1 + I_2,
\end{eqnarray}
where $$I_1 =-\frac{1}{2}\|u^+\|^2_E +\frac{1}{2}\|u^-\|^2_E + C_0 \|u\|^2_{L^{p+1}(\Omega)}$$
 and
$$I_2 = \frac{1}{p+1}\int_\Omega |u|^{p+1} \rho\textrm{d}t \textrm{d}x -C_0 \|u\|^2_{L^{p+1}(\Omega)}.$$

In what follows, we devote to the estimate $I_1$ and $I_2$.

By \eqref{eqa:3.5}, we obtain
\begin{equation}\label{eqa:3.7}
I_1 \geq \frac{1}{2}\|u^+\|^2_E +\frac{1}{2}\|u^-\|^2_E = \frac{1}{2}\|u\|^2_E.
\end{equation}

On the other hand, in virtue of $0<p<1$, for any $\kappa >0$ large enough (which will be chosen later), there exists a constant $\delta_{\kappa}>0$ small enough  such that
\begin{equation}\label{eqa:3.8}
|f(u)| = |u|^p \geq \kappa|u|, \ \ \textrm{for} \ \ |u|\leq \delta_{\kappa}.
\end{equation}
Thus
\[F(u)= \frac{1}{p+1}|u|^{p+1} \geq\left\{\begin{array}{ll}
\frac{\kappa}{p+1}|u|^{2},&\text{\textrm{for} \ $|u|\leq \delta_{\kappa}$},\\
\frac{1}{p+1}|u|^{p+1},& \text{\textrm{for} \ $|u|> \delta_{\kappa}$}.
\end{array}\right.\]
Set
\begin{equation*}
u_1=\left\{\begin{aligned}
u,&\ \ \textrm{if} \ |u|\leq \delta_{\kappa},\\
0,&\ \ \textrm{if}  \ |u|> \delta_{\kappa},
\end{aligned}\right.
\ \ \ \ \ \ \ \
u_2=\left\{\begin{aligned}
0,&\ \ \textrm{if}\ |u|\leq \delta_{\kappa},\\
u,&\ \ \textrm{if} \ |u|> \delta_{\kappa}.
\end{aligned}\right.
\end{equation*}
Then $u= u_1+u_2$ and $\|u\|_{L^{p+1}(\Omega)} \leq \|u_1\|_{L^{p+1}(\Omega)} + \|u_2\|_{L^{p+1}(\Omega)}$. Moreover, reviewing $\Omega = (0, T) \times (0, \pi)$ and $0<\rho(x)\leq\beta_0$, by H\"{o}lder's inequality, we have
\begin{equation*}
\|u_1\|^{p+1}_{L^{p+1}(\Omega)} = \int_\Omega |u_1|^{p+1} \rho\textrm{d}t \textrm{d}x \leq (\beta_0 T\pi)^{\frac{1-p}{2}}\|u_1\|^{p+1}_{L^{2}(\Omega)}.
\end{equation*}
Thus,
\begin{eqnarray*}
I_2 &\geq& \frac{\kappa}{p+1}\int_{\Omega_{\delta_{\kappa}}} |u_1|^{2} \rho\textrm{d}t \textrm{d}x + \frac{1}{p+1}\int_{\Omega\backslash\Omega_{\delta_{\kappa}}} |u_2|^{p+1} \rho\textrm{d}t \textrm{d}x -C_0 \|u\|^2_{L^{p+1}(\Omega)} \\
&=& \frac{\kappa}{p+1} \|u_1\|^{2}_{L^{2}(\Omega)} + \frac{1}{p+1}\|u_2\|^{p+1}_{L^{p+1}(\Omega)} -C_0 \|u\|^2_{L^{p+1}(\Omega)} \\
&\geq& \left(\frac{\kappa}{p+1}(\beta_0 T\pi)^{\frac{p-1}{p+1}} -2C_0\right)\|u_1\|^{2}_{L^{p+1}(\Omega)}+ \left(\frac{1}{p+1}-2C_0\|u_2\|^{1-p}_{L^{p+1}(\Omega)}\right)\|u_2\|^{p+1}_{L^{p+1}(\Omega)},
\end{eqnarray*}
where $\Omega_{\delta_{\kappa}} = \{(t,x)\in \Omega \mid |u|\leq \delta_{\kappa}\}$.

Now, take $\kappa>0$ large enough such that $\frac{\kappa}{p+1}(\beta_0 T\pi)^{\frac{p-1}{p+1}} -2C_0>0$, then fix $\delta_{\kappa}$ satisfying \eqref{eqa:3.8}. In addition, by \eqref{eqa:2.5} we have $\|u_2\|_{L^{p+1}(\Omega)} \leq \|u\|_{L^{p+1}(\Omega)}\leq C\|u\|_E$. Take $R_l>0$ small enough such that $\frac{1}{p+1}-2C_0(CR_l)^{1-p} >0$. Thus, we obtain
$$I_2 >0, \ \ {\rm for} \ \ u\in E_{l+1}\cap S_{R_l}.$$

Consequently, by combining $I_2 >0$ with \eqref{eqa:3.7}, from \eqref{eqa:3.6} we obtain
\begin{equation*}
\Phi(u) \geq \frac{1}{2}\|u\|^2_E = \frac{1}{2} R^2_l, \ \ \textrm{for} \ \ u\in E_{l+1}\cap S_{R_l}.
\end{equation*}
Taking $\sigma_l = \frac{1}{2} R^2_l$, we arrive at the conclusion.
\end{proof}

\begin{proposition}\label{pp:2.2}
$\zeta_{l}\rightarrow 0$ as $l\rightarrow \infty$, where
\begin{equation}\label{eqa:3.a}
\zeta_{l}=\sup_{u\in (E_l)^\perp \backslash \{0\}} \frac{\|u\|_{L^{p+1}(\Omega)}}{\|u\|_E}.
\end{equation}
\end{proposition}
\begin{proof}
In view of $E_l= E^-\oplus(W_l \cap E^+)$, we have $(E_l)^\perp \subset E^+$. Therefore, by Lemma \ref{lem:2.2},
$$\lambda^+_{j_lk_l} = \min\{\lambda_{jk} \in \Lambda(L) \mid \lambda_{jk}>\mu \ {\rm for} \  k>l, \ {\rm or} \ j>l, {\rm or} \ j<-l \}$$
is well defined.

For $u \in (E_l)^\perp\backslash \{0\}$, we have
\begin{equation}\label{eqa:2.9}
\|u\|^2_E = \langle(L-\mu)u, u\rangle= \sum\limits_{j,k}|\lambda_{jk} - \mu| |\alpha_{jk}|^2 \geq (\lambda^+_{j_lk_l} - \mu) \|u\|^2_{L^2(\Omega)}.
\end{equation}

By \eqref{eqa:2.5} and \eqref{eqa:2.9}, with the help of H\"{o}lder's inequality, we obtain
$$\|u\|_{L^{p+1}(\Omega)} \leq \|u\|^\theta_{L^{2}(\Omega)} \|u\|^{1-\theta}_{L^{r}(\Omega)} \leq \frac{C^{1-\theta}}{(\lambda^+_{j_lk_l} - \mu)^{\theta/2}} \|u\|_E,$$
where $r< p+1$, and $1/(p+1)= \theta/2 + (1-\theta)/r$.

Noting $u\in (E_l)^\perp \backslash \{0\}$, it is easy to see that $\lambda^+_{j_lk_l} \rightarrow \infty$ as $l\rightarrow \infty$. Consequently,
$$\zeta_{l}=\sup_{u\in (E_l)^\perp \backslash \{0\}} \frac{\|u\|_{L^{p+1}(\Omega)}}{\|u\|_E} \leq \frac{C^{1-\theta}}{(\lambda^+_{j_lk_l} - \mu)^{\theta/2}} \rightarrow 0, \ \ \textrm{as} \ \ l\rightarrow 0.$$
We arrive at the result.
\end{proof}
\begin{lemma}\label{lem:3.3}
There exists a constant $\varrho_l >0$ such that
$$\Phi(u) \leq \varrho_l, \ \ \forall u\in (E_l)^\perp.$$
In addition, $\varrho_l \rightarrow 0$ as $l \rightarrow\infty$.
\end{lemma}
\begin{proof}
Noting $(E_l)^\perp \subset E^+$, for $u \in (E_l)^\perp$, by \eqref{eqa:3.a} we have
\begin{eqnarray*}
\Phi(u)  &=& -\frac{1}{2}\langle(L-\mu)u, u\rangle + \frac{1}{p+1}\int_\Omega |u|^{p+1} \rho\textrm{d}t \textrm{d}x \nonumber \\
 &\leq& -\frac{1}{2} \|u\|^2_E + \frac{1}{p+1} \zeta^{p+1}_{l}\|u\|^{p+1}_E.
\end{eqnarray*}
Since $0<p<1$, a direct calculation yields
\begin{eqnarray*}
\Phi(u)  &\leq& -\frac{1}{2} \zeta^{\frac{2(p+1)}{1-p}}_{l} + \frac{1}{p+1} \zeta^{p+1}_{l} \zeta^{\frac{(p+1)^2}{1-p}}_{l}\nonumber \\
&=& \Big( \frac{1}{p+1} -\frac{1}{2}\Big)\zeta^{\frac{2(p+1)}{1-p}}_{l}.
\end{eqnarray*}
Let $\varrho_l =\Big( \frac{1}{p+1} -\frac{1}{2}\Big)\zeta^{\frac{2(p+1)}{1-p}}_{l}$, by proposition \ref{pp:2.2}, we have $\varrho_l \rightarrow 0$ as $l \rightarrow\infty$. The proof is completed.
\end{proof}

\section{Sequence of critical points for restricted functional}
\label{sec:4}

In this section, we intend to employ the $\mathbb{Z}_2$-index theory and minimax method to obtain a sequence of critical points for the functional $\Phi$ restricted on some given subspaces of $E$.

At first, in order to make the following statement more precise, we make some notations.  Denote
$$\Sigma=\{A\subset E^m\backslash \{0\} \mid A =\bar{A}, -A=A\}$$
which is made of  closed symmetric subset of $E^m\backslash \{0\}$, where $\bar{A}$ denotes the closure of $A$.
For any $c\in \mathbb{R}$, let $\Phi^{c} = \{u\in E \mid \Phi(u) \geq c\}$ denote the level set of $\Phi$ at $c$, and
$$K = \{u \in E^m \mid \Phi_m' (u)= 0\}$$
denote the critical points set of $\Phi_m=\Phi|_{E^m}$.

The mapping $\gamma:\Sigma \rightarrow \mathbb{N}\cup\{+\infty\}$ is called  genus  if it satisfies, for any $A\in \Sigma$,
\begin{eqnarray*}\gamma(A)=
\begin{cases}
0, \quad\textrm{\rm if} \ A=\emptyset,\cr \inf\{N \in \mathbb{Z_+}| \exists \ {\rm an \ odd \ mapping} \ h\in C(A, \mathbb{R}^N\backslash\{0\}) \},  \cr +\infty, \quad \textrm{\rm if no such odd mapping}.\end{cases}
\end{eqnarray*}

In this paper, we need to use the following properties of genus (see \cite{Chang.1986}).\\
($1^\circ$) Assume $E =E_1\oplus E_2$, $\dim E_1 = N$, and $\gamma(A) >N$ for $A\in \Sigma$, then $A\cap E_2 \neq \emptyset$.\\
($2^\circ$) Assume that $U\subset \mathbb{R}^N$ is an open bounded symmetric neighborhood of the origin in $\mathbb{R}^N$, then $\gamma(\partial U) =N$.\\
($3^\circ$) (Super-variant) For any continuous odd mapping $h : E^m \rightarrow E^m$, it holds $\gamma(A)\leq \gamma(\overline{h(A)})$, $\forall A \in \Sigma$.

The following compact embedding plays an important role in this paper.
\begin{proposition}\label{pp:2.1}
For any $r \in [1, 2]$, the embedding
\begin{equation}\label{eqa:2.8}
E\ominus E_0 \hookrightarrow L^r(\Omega),
\end{equation}
is compact.
\end{proposition}
\begin{proof}
For $u \in E\ominus E_0$, we write $u = \sum\limits_{ka \neq |j|b} \alpha_{jk} \phi_j\varphi_k$.

Noting $\|u\|_E = \big(\sum\limits_{j,k}|\lambda_{jk} - \mu| |\alpha_{jk}|^2\big)^\frac{1}{2}$, we have that the mapping
$$\tau_0: u = \sum\limits_{j,k} \alpha_{jk} \phi_j\varphi_k \mapsto \{|\lambda_{jk} - \mu|^\frac{1}{2} \alpha_{jk}\}$$
is continuous from $E\ominus E_0$ to $l^2$.

Observing that $|\lambda_{jk} - \mu|\rightarrow \infty$ as $j, k \rightarrow \infty$, it follows that the mapping
$$\tau_1:\{|\lambda_{jk} - \mu|^\frac{1}{2}\alpha_{jk}\} \mapsto \{\alpha_{jk}\}$$
is compact from $l^2$ to $l^2$.

Since $\phi_j\varphi_k$ is a complete orthonormal sequence of $L^2(\Omega)$, then the mapping
$$\tau_2:\{\alpha_{jk}\} \mapsto u= \sum\limits_{j,k} \alpha_{jk} \phi_j\varphi_k$$
is continuous from $l^2$ to $L^2(\Omega)$.

Consequently, the mapping
$$ \tau_2\tau_1\tau_0:E\ominus E_0 \rightarrow L^2(\Omega)$$
is compact. Furthermore, for $1 \leq r \leq 2$, the continuous embedding $L^2(\Omega) \hookrightarrow L^r(\Omega)$ implies that the embedding $E\ominus E_0\hookrightarrow L^r(\Omega)$ is compact.
\end{proof}
To acquire the critical points of $\Phi_m=\Phi|_{E^m}$ on subspaces $E^m$ by variational methods, it needs to verify that $\Phi_m$ satisfies $(PS)$ condition, which means that any sequence $\{u_i\} \subset E^m$ for which $\Phi_m (u_i)$ is bounded and $\Phi_m' (u_i) \rightarrow 0$ as $i \rightarrow \infty$ contains a convergent subsequence.
\begin{lemma}\label{lem:3.1}

$\Phi_m$ satisfies $(PS)$ condition.
\end{lemma}
\begin{proof}
Assume  $\{u_i\} \subset E^m$ for which $\Phi_m (u_i)$ is bounded and $\Phi_m' (u_i) \rightarrow 0$ as $i \rightarrow \infty$.
Split $u_i = u^+_i +u^-_i$, where $u^+_i \in E^+$ and $u^-_i \in W_m \cap E^-$.

For $u^+_i \in E^+$, since $\Phi' (u_i) \rightarrow 0$ as $i \rightarrow \infty$, by \eqref{eqa:2.7} and \eqref{eqa:2.b}, we have
\begin{eqnarray}\label{eqa:3.1}
o(1)\|u^+_i\|_E &\geq & \langle -\Phi'( u_i), u^+_i \rangle = \langle (L - \mu)u^+_i, u^+_i \rangle - \int_\Omega |u_i|^{p-1}u_iu^+_i \rho\textrm{d}t \textrm{d}x \nonumber\\
&\geq& \|u^+_i\|^2_E - \|u_i\|^p_{L^{p+1}(\Omega)}\|u^+_i\|_{L^{p+1}(\Omega)}.\ \ \ \ \ \ \
\end{eqnarray}
By combining \eqref{eqa:2.5} with \eqref{eqa:3.1}, a direct calculation yields
\begin{equation}\label{eqa:3.2}
\|u^+_i\|_E - C \|u_i\|^p_{E} \leq o(1),
\end{equation}
where the constant $C$ is independent of $i$.

For $u^-_i \in W_m \cap E^-$, by \eqref{eqa:2.7} and \eqref{eqa:2.a} we have
\begin{eqnarray*}
o(1)\|u^-_i\|_E &\geq & \langle \Phi'( u_i), u^-_i \rangle = -\langle (L - \mu)u^-_i, u^-_i \rangle + \int_\Omega |u_i|^{p-1}u_iu^-_i \rho\textrm{d}t \textrm{d}x \\
&\geq& \|u^-_i\|^2_E - \|u_i\|^p_{L^{p+1}(\Omega)}\|u^-_i\|_{L^{p+1}(\Omega)}.\ \ \ \ \ \ \
\end{eqnarray*}
Similarly, we have
\begin{equation}\label{eqa:3.3}
\|u^-_i\|_E - C \|u_i\|^p_{E} \leq o(1).
\end{equation}
Since $0<p<1$ and $\|u_i\|^2_E = \|u^+_i\|^2_E + \|u^-_i\|^2_E$, by \eqref{eqa:3.2} and \eqref{eqa:3.3}, we have that there exists a constant $M>0$ independent of $i$ such that
$$\|u_i\|_E \leq M.$$

On the other hand, since $E$ is a Hilbert space, we have $u_i \rightharpoonup u$ in $E$ along with a subsequence as $i \rightarrow \infty$ for some $u \in E^m$. For the sake of convenience, we still use $\{u_i\}$ to denote the subsequence. Let $u^-$ denote the weak limit of $\{u^-_i\}$, and decompose $u_i^+ = v_i + y_i$ and $u^+ = v + y$,  where $v$, $y$ are the weak limits of $\{v_i\}$,  $\{y_i\}$ respectively, and $v_i, v \in E^+ \ominus E_0$, $y_i, y \in E^+ \cap E_0$.

In what follows, we devote to the proof $u_i \rightarrow u$ in $E^m$.

In virtue of $\dim(W_m \cap E^-) < \infty$, it follows $u^-_i \rightarrow u^-$ in $E$.

Noting $v_i, v \in E^+ \ominus E_0$ and $u_i = v_i + y_i + u^-_i$, we have $u_i - v_i\in (E^+ \ominus E_0)^\perp$. Thus, we have
$$\langle (L - \mu)( u_i-v_i), v_i-v \rangle = 0.$$
Therefore, we have
\begin{eqnarray}\label{eqa:3.4}
&&\|v_i-v\|^2_E = \langle (L - \mu)( v_i-v), v_i-v \rangle \nonumber\\
&=& -\langle \Phi'(u_i), v_i-v \rangle + \int_\Omega |u_i|^{p-1}u_i (v_i-v) \rho\textrm{d}t \textrm{d}x - \langle (L - \mu) v, v_i-v \rangle\nonumber\\
&\leq &  o(1)\|v_i-v\|_E+ \|u_i\|^p_{L^{p+1}(\Omega)}\|v_i-v\|_{L^{p+1}(\Omega)} + o(1),
\end{eqnarray}
Since $v_i, v \in E^+ \ominus E_0$, by \eqref{eqa:2.8}, we have $v_i \rightarrow v$ in $L^{p+1}(\Omega)$ as $i\rightarrow \infty$. Therefore, from \eqref{eqa:3.4}, we obtain
\begin{eqnarray*}
\|v_i-v\|_E \rightarrow 0,  \ \ \textrm{as} \ \ i\rightarrow 0.
\end{eqnarray*}

Since $\dim(E^+ \cap E_0) < \infty$, it follows $y_i \rightarrow y$ in $E$. We complete the proof.
\end{proof}

Recalling $\dim E^m_{l+1} <\infty$, denote
$$\mathscr{F}_l^m = \{A\in \Sigma \mid A\subset E^m, \gamma(A) \geq \dim E^m_{l+1}\},$$
and
\begin{equation}
c_{lm} = \sup_{A\in \mathscr{F}_l^m} \inf_{u\in A} \Phi_m (u), \quad \forall m\in\mathbb{N}^+.
\end{equation}

By the properties of genus and deformation lemma argument (see \cite{Chang.1986}), we prove the following lemma.
\begin{lemma}\label{lem:3.a}
 $c_{lm}$ is a critical value of $\Phi_m$ and satisfies
$$\sigma_l \leq c_{lm} \leq \varrho_l.$$
\end{lemma}
\begin{proof}
Firstly, it is proved by contradiction. If $c_{lm}$ is not a critical value of $\Phi_m$, then there exists $\bar{\varepsilon} >0$ such that $\Phi_m^{-1} [c_{lm}-\bar{\varepsilon}, c_{lm}+\bar{\varepsilon}]\cap K = \emptyset$, provided by the $(PS)$ condition.

By the definition of $c_{lm}$, for any $\varepsilon \in (0, \bar{\varepsilon})$, there exists $A_0 \in \mathscr{F}_l^m$ such that
$\inf_{u\in A_0} \Phi_m (u) \geq c_{lm} - \varepsilon$ which implies
$$A_0 \subset \Phi^{c_{lm} - \varepsilon}_m.$$

Since $\Phi_m \in C^1(E^m, \mathbb{R})$ is even and satisfies $(PS)$ condition, then there exists odd map $\eta \in C([0, 1]\times E^m, E^m)$ satisfying:\\
(a) $\eta(t, \cdot)$ is a homeomorphism of $E^m$, $\forall t\in[0,1]$;\\
(b) $\eta(1, \Phi^{c_{lm} - \varepsilon}_m) \subset \Phi^{c_{lm} + \varepsilon}_m$.

By the properties ($3^\circ$) and (a),  we obtain $\gamma (A_0) \leq \gamma (\eta (1, A_0))$. Thus, we have $\eta (1, A_0) \in \mathscr{F}_l^m$. Noting $A_0 \subset \Phi^{c_{lm} - \varepsilon}_m$, by the property (b) it follows
$$c_{lm} \geq \inf_{u \in \eta (1, A_0)} \Phi_m (u) \geq c_{lm} + \varepsilon$$
which is a contradiction. Therefore, $c_{lm}$ is a critical value of $\Phi_m$.

Secondly, since $\dim E^m_{l}< \dim E^m_{l+1} <\infty$, then for $A\in \mathscr{F}_l^m$, by property ($1^\circ$) we obtain $A\cap (E^m_{l})^\perp \neq \emptyset$. Since $(E^m_{l})^\perp = (E^m)^\perp \cup (E_{l})^\perp$ and $A \subset E^m$, we have
$$A\cap (E_{l})^\perp = A\cap (E^m_{l})^\perp  \neq \emptyset.$$
Thus, from Lemma \ref{lem:3.3}, we have
$$\inf_{u\in A} \Phi_m (u) \leq \sup_{u \in A\cap (E_{l})^\perp} \Phi_m (u) \leq \sup_{u \in (E_{l})^\perp} \Phi_m (u) \leq\varrho_l, \quad \forall A\in \mathscr{F}_l^m.$$
Consequently, $c_{lm} = \sup_{A\in \mathscr{F}_l^m} \inf_{u\in A} \Phi_m (u) \leq \varrho_l$.

On the other hand, denote $S^m_{R_l} = E^m_{l+1}\cap S_{R_l}$, where ${R_l}$ is the constant occurred in Lemma \ref{lem:3.2}. Thus, by property ($2^\circ$) we have $\gamma(S^m_{R_l}) =\dim E^m_{l+1}$. Thus,  $S^m_{R_l} \in \mathscr{F}_l^m$. By Lemma \ref{lem:3.2}, we obtain
$$c_{lm} = \sup_{A\in \mathscr{F}_l^m} \inf_{u\in A} \Phi_m (u)\geq \inf_{u\in S^m_{R_l}}\Phi_m (u) \geq \inf_{u\in E_{l+1}\cap S_{R_l}}\Phi_m (u) \geq \sigma_l.$$
We complete the proof.
\end{proof}

Let $\{u_{lm}\}$ be the sequence of critical points of  $\Phi_m$  corresponding to the critical values $c_{lm}$. Whereafter, we shall to prove that $\{u_{lm}\}$ contains a subsequence as $m$ goes to infinity.

\section{The proof of Theorem \ref{th:2.1}}
\label{sec:5}

In this section, we first prove  for every $l \in  \mathbb{N}^+$, $u_{lm} \rightarrow u_l$ in $E$ as $m \rightarrow \infty$ for some $u_l\in E$. Then we prove that $u_l$ are the critical points of the functional $\Phi$.

\begin{lemma}\label{lem:4.1}
There exists a constant $M_0 >0$ independent of $m$ such that $\|u_{lm}\|_{E} \leq M_0$.
\end{lemma}
\begin{proof}
Since $u_{lm}$ are the critical points of $\Phi_m$, we have $\Phi_m' (u_{lm}) =0$. Thus,
\begin{equation}\label{eqa:4.1}
\langle (L - \mu)u_{lm}, v \rangle = \int_\Omega |u_{lm}|^{p-1}u_{lm}v \rho\textrm{d}t \textrm{d}x, \ \ \forall v\in E^m.
\end{equation}
Taking $v=u_{lm}$ in \eqref{eqa:4.1}, by Lemma \ref{lem:3.a}, we have
\begin{equation*}
\varrho_l \geq \Phi_m (u_{lm}) = \left(\frac{1}{p+1} - \frac{1}{2}\right)\int_\Omega |u_{lm}|^{p+1} \rho\textrm{d}t \textrm{d}x.
\end{equation*}
Thus, there exists a constant $C>0$ independent of $m$ such that
\begin{equation}\label{eqa:4.a}
\|u_{lm}\|_{L^{p+1}(\Omega)} \leq C.
\end{equation}
Split $u_{lm} = u^+_{lm} + u^-_{lm}$, where $u^+_{lm}\in E^m\cap E^+$ and $u^-_{lm}\in E^m\cap E^-$. Noting $\|u^+_{lm}\|^2_{E}=\langle (L - \mu)u^+_{lm}, u^+_{lm} \rangle$ and taking $v=u_{lm}$ in \eqref{eqa:4.1}, by \eqref{eqa:2.5} and \eqref{eqa:4.a}, we obtain
$$\|u^+_{lm}\|^2_{E}=\langle (L - \mu)u_{lm}, u^+_{lm} \rangle = \int_\Omega |u_{lm}|^{p-1}u_{lm}u^+_{lm} \rho\textrm{d}t \textrm{d}x \leq C \|u^+_{lm}\|_{E}.$$
Therefore, we obtain  $\|u^+_{lm}\|_{E} \leq C$. Similarly, $\|u^-_{lm}\|_{E} \leq C$. Consequently, there exists a constant $M_0 >0$ independent of $m$ such that $\|u_{lm}\|_{E} \leq M_0$. The proof is completed.
\end{proof}

Since $E$ is a Hilbert space, by Lemma {\upshape\ref{lem:4.1}}, we have $u_{lm} \rightharpoonup u_l$ in $E$ along with a subsequence as $m \rightarrow \infty$ for some $u_l \in E$. The following lemma shows that we can extract a subsequence of $\{u_{lm}\}$ which converges strongly to $u_l \in E$.

\begin{lemma}\label{lem:4.2}
$u_{lm} \rightarrow u_l$ in $E$ along with a subsequence as $m \rightarrow \infty$ for some $u_l \in E$.
\end{lemma}
\begin{proof}
Decompose $u_{lm} = v_{lm} + y_{lm} + w_{lm} + z_{lm}$ and $u_l = v_l + y_l + w_l + z_l$,  where $v_l$, $y_l$, $w_l$, $z_l$ are the weak limits of $\{v_{lm}\}$,  $\{y_{lm}\}$, $\{w_{lm}\}$, $\{z_{lm}\}$ respectively, and $v_{lm}, v_l \in E^+ \ominus E_0$, $y_{lm}, y_l \in E^+ \cap E_0$, $w_{lm}, w_l\in E^- \ominus E_0$, $z_{lm}, z_l \in E^- \cap E_0$.

Recalling $E^m \subset E^{m+1}$ and $E=\bigcup_{m\in\mathbb{N}} E^m$, we have
$$\|(id - P_m)u\|_E \rightarrow 0, \ \ \textrm{as} \ \ m\rightarrow \infty, \ \ \forall u\in E,$$
where $id$ denotes the identity mapping and $P_m : E \rightarrow E_m$ is the natural projection. In addition, in virtue of \eqref{eqa:2.5}, we obtain
\begin{equation}\label{eqa:4.3}
\|(id - P_m)u\|_{L_{p+1}(\Omega)} \rightarrow 0, \ \ \textrm{as} \ \ m\rightarrow \infty, \ \ \forall u\in E.
\end{equation}

(i) For $v_{lm}, v_l \in E^+ \ominus E_0$, by \eqref{eqa:2.b}, we have
\begin{eqnarray}\label{eqa:4.4}
\|v_{lm}-v_l\|^2_E &=& \langle (L - \mu)( v_{lm}-v_l), v_{lm}-v_l \rangle \nonumber\\
&=& \langle (L - \mu) v_{lm}, v_{lm}-v_l \rangle -\langle (L - \mu) v_l, v_{lm}-v_l \rangle.
\end{eqnarray}
In virtue of $v_{lm} \rightharpoonup v_l$ in $E$ as $m \rightarrow \infty$ and $E \hookrightarrow L^2(\Omega)$, we have $v_{lm} \rightharpoonup v_l$ in $L^2(\Omega)$. Furthermore, with the aid of Proposition \ref{pp:2.1}, we obtain $v_{lm} \rightarrow v_l$ in $L^{p+1}(\Omega)$ as $m \rightarrow \infty$. Thus, it follows
\begin{eqnarray}\label{eqa:4.5}
\langle (L - \mu) v_l, v_{lm}-v_l \rangle \rightarrow 0, \ \ {\rm as} \ \ m \rightarrow \infty.
\end{eqnarray}

On the other hand, noting $v_{lm}, v_l \in E^+ \ominus E_0$ and $u_{lm} = v_{lm} + y_{lm} + w_{lm} + z_{lm}$, we have $u_{lm} - v_{lm}\in (E^+ \ominus E_0)^\perp$. Thus
$$\langle (L - \mu)( u_{lm}-v_{lm}), v_{lm}-v_l \rangle = 0.$$
Since $u_{lm} \in E^m$, and $(id - P_m)v_l \in (E^m)^\perp$ and $v_{lm} - P_mv_l \in E^m$, from \eqref{eqa:4.1} we have
\begin{eqnarray}\label{eqa:4.6}
&&\langle (L - \mu) v_{lm}, v_{lm}-v_l \rangle = \langle (L - \mu) u_{lm}, v_{lm}-v_l \rangle \nonumber\\
&=& \langle (L - \mu) u_{lm}, v_{lm}- P_m v_l \rangle + \langle (L - \mu) u_{lm}, P_m v_l-v_l \rangle \nonumber\\
&=& \int_\Omega |u_{lm}|^{p-1}u_{lm}(v_{lm}-P_mv_l) \rho\textrm{d}t \textrm{d}x\nonumber\\
&\leq &  \|u_{lm}\|^p_{L^{p+1}(\Omega)}\|v_{lm}-v_l\|_{L^{p+1}(\Omega)} + \|u_{lm}\|^p_{L^{p+1}(\Omega)}\|(id-P_m)v_l\|_{L^{p+1}(\Omega)}\nonumber\\
&\leq & o(1),
\end{eqnarray}
where the last inequality is acquired by \eqref{eqa:4.3} and $v_{lm} \rightarrow v_l$ in $L^{p+1}(\Omega)$.

Inserting \eqref{eqa:4.5} and \eqref{eqa:4.6} into \eqref{eqa:4.4}, we have
\begin{eqnarray}\label{eqa:4.7}
\|v_{lm}-v_l\|_E \rightarrow 0,  \ \textrm{as} \ \ m \rightarrow \infty.
\end{eqnarray}

(ii) Similarly, for  $w_{lm}, w_l\in E^- \ominus E_0$, we have
\begin{equation}\label{eqa:4.8}
\|w_{lm}-w_l\|^2_E = -\langle (L - \mu)( w_{lm}-w_l), w_{lm}-w_l \rangle  \rightarrow 0,   \ \textrm{as} \ \ m \rightarrow \infty.
\end{equation}

(iii) For $y_{lm}, y_l \in E^+ \cap E_0$, Remark \ref{rem:2.1} shows $\dim (E^+ \cap E_0) <\infty$. Thus, $y_{lm} \rightharpoonup y_l$ in $E$ implies
\begin{eqnarray}\label{eqa:4.9}
\|y_{lm}-y_l\|_E \rightarrow 0,  \ \textrm{as} \ \ m \rightarrow \infty.
\end{eqnarray}

(iv) For $z_{lm}, z_{l} \in E^- \cap E_0$, since the compact embedding $E\ominus E_0 \hookrightarrow L^q(\Omega)$ is not holding  and $\dim (E^- \cap E_0) = \infty$, then the compact argument is invalid. In what follows, by the monotone method, we will prove $z_{lm}\rightarrow z_l$ in $E$ as $m \rightarrow \infty$.

Since  $z_{lm} \rightharpoonup z$ in $E$, we have
\begin{eqnarray*}
&&\|z_{lm}-z\|^2_E = - \langle (L - \mu)( z_{lm}-z_l), z_{lm}-z_l \rangle \nonumber\\
&=& - \langle (L - \mu) u_{lm}, z_{lm}-P_mz_l \rangle + \langle (L - \mu) z_l, z_{lm}-z_l \rangle\nonumber\\
&\leq &  - \langle f(u_{lm}), (id-P_m)z_l \rangle -  \langle f(u_{lm}), z_{lm}-z_l \rangle+ o(1).
\end{eqnarray*}
In virtue of \eqref{eqa:4.a} and \eqref{eqa:4.3}, we have
$$|\langle f(u_{lm}), (id-P_m)z_l \rangle| \leq \|u_{lm}\|^p_{L^{p+1}(\Omega)}\|(id-P_m)z_l\|_{L^{p+1}(\Omega)} \rightarrow 0, \ \ \textrm{as} \ \ m\rightarrow \infty.$$

On the other hand, decompose
\begin{eqnarray*}
\langle f(u_{lm}), z_{lm}-z_l \rangle &=& \langle f(u_{lm})- f(\tilde{u}_{lm} + z_l), z_{lm}-z_l \rangle\nonumber\\
&+& \langle f(\tilde{u}_{lm} + z_l) - f(u_l), z_{lm}-z_l \rangle + \langle f(u_l), z_{lm}-z_l \rangle,
\end{eqnarray*}
where $\tilde{u}_{lm} = v_{lm} + y_{lm} + w_{lm}$.

Since $f(u)= |u|^{p-1}u$ is nondecreasing in $u$, then
\begin{equation*}
\langle f(u_{lm})- f(\tilde{u}_{lm} + z_l), z_{lm}-z_l \rangle \geq 0.
\end{equation*}

By \eqref{eqa:2.4} and  {\eqref{eqa:4.7}--\eqref{eqa:4.9}}, we have $\tilde{u}_{lm} \rightarrow \tilde{u_l}$ in $L^2(\Omega)$, where $\tilde{u_l} = v_l + y_l + w_l$. Moreover, since $f : u \mapsto |u|^{p-1}u$ is continuous from ${L^2(\Omega)}$ to ${L^{2/p}(\Omega)}$. Thus, we have
\begin{equation*}
\langle f(\tilde{u}_{lm} + z) - f(u_l), z_{lm}-z_l \rangle \rightarrow 0, \ \ {\rm as} \ \ m \rightarrow \infty.
\end{equation*}
In addition, in virtue of  $z_{lm} \rightharpoonup z_l$ in ${L^2(\Omega, \rho)}$, we have
\begin{equation*}
\langle f(u_l), z_{lm}-z_l \rangle \rightarrow 0, \ \ {\rm as} \ \ m \rightarrow \infty.
\end{equation*}

Consequently,
\begin{equation}\label{eqa:4.10}
\|z_{lm}-z_l\|^2_E \rightarrow 0,  \ \textrm{as} \ \ m \rightarrow \infty.
\end{equation}

Finally, from {\eqref{eqa:4.7}--\eqref{eqa:4.10}}, we obtain $\|u_{lm}-u_l\|_E \rightarrow 0$, as $m \rightarrow \infty$.
\end{proof}

Now, we give the proof of Theorem \ref{th:2.1}.
\begin{proof}
%We devote to the proof that $u_l$ are the weak solutions of problem \eqref{eqa:2-1}--\eqref{eqa:2-3}, i.e.,
%\begin{equation}\label{eqa:4.11}%$\langle(L-\mu)u_l, \psi \rangle - \int_\Omega f(u_l)\psi \rho\textrm{d}t \textrm{d}x =0, \ \ \forall \psi\in E,
%\end{equation}
%where $f(u_l)=|u_l|^{p-1}u_l$.

Since $u_{lm}$ are the critical points of  $\Phi_m$, a simple calculation yields
$$\langle (L - \mu) u_{lm}, u_{lm}-v \rangle = \langle f(u_{lm}), u_{lm}-P_mv \rangle,  \ \ \forall v\in E.$$
Thus, with the aid of monotonicity of $f(u)$, we have
\begin{eqnarray}\label{eqa:4.12}
\langle (L - \mu) u_{lm}, u_{lm}-v \rangle - \langle f(v), u_{lm}-v \rangle
&=&\langle f(u_{lm}), u_{lm}-P_mv \rangle - \langle f(v), u_{lm}-v \rangle \nonumber\\
&\geq & \langle f(u_{lm}), v -P_mv\rangle.
\end{eqnarray}

By \eqref{eqa:4.a} and \eqref{eqa:4.3}, we have
$$|\langle f(u_{lm}), v -P_mv \rangle| \leq \|u_{lm}\|^p_{L^{p+1}(\Omega)}\|(id-P_m)v\|_{L^{p+1}(\Omega)} \rightarrow 0, \ \ \textrm{as} \ \ m\rightarrow \infty.$$
Thus, noting $u_{lm} \rightarrow u_l$ in $E$ as $m\rightarrow \infty$ and passing to the limit in \eqref{eqa:4.12}, a simple calculation yields that for any $v\in E$,
\begin{eqnarray}\label{eqa:4.13}
0&=& \lim_{m\rightarrow \infty}\langle f(u_{lm}), v -P_mv\rangle\nonumber\\
&\leq & \lim_{m\rightarrow \infty} \langle (L - \mu) u_{lm}, u_{lm}-v \rangle - \lim_{m\rightarrow \infty}\langle f(v), u_{lm}-v \rangle\nonumber\\
&= & \langle (L - \mu) u_{l}, u_{l}-v \rangle - \langle f(v), u_{l}-v \rangle.
\end{eqnarray}
Taking $v= u_l -s \psi$ for $s>0$, $\psi \in E$ in \eqref{eqa:4.13}, and dividing by $s$, we obtain
$$\langle (L - \mu) u_{l}, \psi \rangle - \langle f(u_l -s \psi), \psi \rangle\geq 0.$$
Now letting $s \rightarrow 0$ in above inequality, we have
$$\langle (L - \mu) u_{l}, \psi \rangle - \langle f(u_l), \psi \rangle\geq 0.$$
Since $\psi$ is chosen arbitrarily,  then
$$\langle (L - \mu) u_{l}, \psi \rangle - \langle f(u_l), \psi \rangle = 0,$$
i.e.,
\begin{equation*}
\langle(L-\mu)u_l, \psi \rangle - \int_\Omega f(u_l)\psi \rho\textrm{d}t \textrm{d}x =0, \ \ \forall \psi\in E.
\end{equation*}

On the other hand, noting $u_{lm} \rightarrow u_l$ in $E$ as $m\rightarrow \infty$ and by \eqref{eqa:2.5}, we have
\begin{equation*}
\Phi(u_{lm}) \rightarrow \Phi(u_{l}),  \ \ \textrm{as} \ \ m\rightarrow \infty,
\end{equation*}
where $\Phi(u)$  is defined by \eqref{eqa:2.6}.
Therefore, by Lemma \ref{lem:3.a}, we obtain
$$0<\sigma_l \leq \Phi(u_{l}) \leq \varrho_l.$$
Since $\varrho_l \rightarrow 0$ as $l \rightarrow\infty$, then $\{u_{l}\}$ are infinitely many weak solutions of problem \eqref{eqa:2-a}--\eqref{eqa:2-c}. In addition, since $\Phi'(u_{l}) = 0$ and $0<p<1$, by $\varrho_l \rightarrow 0$ as $l \rightarrow\infty$, a simple calculation yields
$$\int_\Omega |u_{l}|^{p+1} \rho\textrm{d}t \textrm{d}x =\left(\frac{1}{p+1} - \frac{1}{2}\right)^{-1} \Phi(u_{l}) \rightarrow 0, \ \ \textrm{as} \ \ l \rightarrow\infty.$$
The proof is completed.
\end{proof}

\vskip 5mm

{\bf Ethics statement.} This work did not involve any active collection of human data.

\vskip 5mm

{\bf Data accessibility statement.} This work does not have any experimental data.

\vskip 5mm

{\bf Competing interests statement.} We have no competing interests.

\vskip 5mm

{\bf Authors¡¯ contributions.}
HW and SJ contributed to the mathematical proof and writing the paper. All authors gave final approval for publication.

\vskip 5mm

{\bf Funding.} This work was supported by NSFC Grants (nos. 11322105 and 11671071).



%% References


\bibliographystyle{elsarticle-num}
\bibliography{<your-bib-database>}



\end{document}